\documentclass[reqno,11pt]{amsart}
\usepackage[english]{babel}
\usepackage{amssymb,verbatim,microtype,graphicx,hyperref,cite}
\usepackage[T1]{fontenc}

\newtheorem{theorem}{Theorem}[section]
\newtheorem{proposition}[theorem]{Proposition}
\newtheorem{lemma}[theorem]{Lemma}

\newtheorem{corollary}[theorem]{Corollary}

\theoremstyle{definition}
\newtheorem{definition}[theorem]{Definition}

\newtheorem{remark}[theorem]{Remark}

\newtheorem*{ackn}{Acknowledgements}

\numberwithin{equation}{section}

\newcommand{\C}{\mbox{$\mathbb{C}$}}

\newcommand{\Capo}{\mbox{Cap}_{m}}

\newcommand{\Ker}{\mbox{$\mathcal{N}_{m}$}}
\newcommand{\E}{\mbox{$\mathcal{E}_{m}$}}
\newcommand{\Eo}{\mbox{$\mathcal{E}^{0}_{m}$}}

\newcommand{\F}{\mbox{$\mathcal{F}_{m}$}}

\newcommand{\HH}{\operatorname{H}_{m}}
\newcommand{\MSH}[1]{\mbox{$\mathcal{SH}_{m}(#1)$}}
\newcommand{\PP}{\mbox{\textbf{P}}}

 \usepackage{hyperref}
\hypersetup{
    unicode=false,
    pdftoolbar=true,
    pdfmenubar=true,
    pdffitwindow=false,
    pdfstartview={FitH},
    pdftitle={Kiselman minimum principle and Rooftop envelopes in Complex Hessian Equations},
    pdfauthor={Ahag, Czyz, Lu and Rashkovskii},
    colorlinks=true,
   linkcolor=black,
    citecolor=black,
    filecolor=black,
    urlcolor=black}

\begin{document}

\title[Kiselman Minimum Principle and Rooftop Envelopes]{Kiselman Minimum Principle and Rooftop Envelopes in Complex Hessian Equations}

\author{Per \AA hag}\address{Department of Mathematics and Mathematical Statistics\\ Ume\aa \ University\\SE-901~87 Ume\aa \\ Sweden}\email{per.ahag@umu.se}

\author{Rafa\l\ Czy{\.z}}\address{Faculty of Mathematics and Computer Science, Jagiellonian University, \L ojasiewicza~6, 30-348 Krak\'ow, Poland}
\email{rafal.czyz@im.uj.edu.pl}

\author{Chinh H. Lu}\address{Laboratoire Angevin de Recherche en Math\'{e}matiques (LAREMA), Universit\'{e} d'Angers (UA), 2 Boulevard de Lavoisier, 49000 Angers, France}
\email{hoangchinh.lu@univ-angers.fr}

\author{Alexander Rashkovskii}\address{Department of Mathematics and Physics, University of Stavanger, 4036 Stavanger, Norway}\email{alexander.rashkovskii@uis.no}

\date{\today}

\begin{abstract}
We initiate the study of $m$-subharmonic functions with respect to a semipositive $(1,1)$-form in Euclidean domains, providing a significant element in understanding geodesics within the context of complex Hessian equations. Based on the foundational Perron envelope construction, we prove a decomposition of $m$-subharmonic solutions, and a general comparison principle that effectively manages singular Hessian measures. Additionally, we establish a rooftop equality and an analogue of the Kiselman minimum principle, which are crucial ingredients in establishing a criterion for geodesic connectivity among $m$-subharmonic functions, expressed in terms of their asymptotic envelopes.
\end{abstract}

\maketitle

\tableofcontents

\section{Introduction}

Since the German mathematician Oskar Perron introduced his celebrated envelope construction in 1923 to solve a boundary value problem for the Laplace equation~\cite{Perron1923}, the method has not only stood the test of time but has also gained prominence in contemporary mathematical analysis. In this paper, given a non-negative regular Borel measure $\mu$, we consider $m$-subharmonic solutions to the complex Hessian equation:
\[
H_m(u)=\mu,
\]
where $1\leq m\leq n$. Here, the Hessian operator is defined by $H_m(u)=(dd^c u)^m \wedge (dd^c |z|^2)^{n-m}$ for smooth $u$, and extends for bounded $m$-subharmonic functions by Bedford-Taylor theory \cite{BT76}, as presented by B{\l}ocki.  These equations, which interpolate between the Laplace and Monge-Ampère equations, have been intensively studied over the past twenty years by Li \cite{Li04}, Błocki \cite{Blo05}, and many others. We refer the reader to Section~\ref{Sec:Preliminaries} for the necessary background.

This paper aims to introduce, develop, and employ \emph{rooftop envelopes} for $m$-subharmonic functions, thus carrying forward the enduring legacy of Perron's original envelope construction. Our inspiration comes from the pluricomplex counterpart, the plurisubharmonic rooftop envelopes, recently utilized in  multiple papers, including \cite{DDL5,GL21a, GL22, GL23Crelle,LN22,Sal23}.

In classical potential theory, the decomposition of subharmonic functions into regular and singular parts corresponds uniquely to the decomposition of their associated Laplace measures, $\Delta u = \mu_r + \mu_s$. Here, $\mu_s$ is carried by the polar set $\{u = -\infty\}$, and $\mu_r$ vanishes on these sets. Such a decomposition highlights the linearity of the Laplace operator, allowing for a straightforward decomposition of the function $u = u_r + u_s$, where $\Delta u_r = \mu_r$ 
and $\Delta u_s = \mu_s$.

However, the complex Hessian operator's non-linearity precludes a direct analogous decomposition for $m$-subharmonic functions. To address this, we develop an adapted approach, presented in the following result:

\bigskip

\noindent {\bf Theorem~A}. \emph{
For any $u \in \Ker$, there exist unique functions $u_r, u_s \in \Ker$ satisfying the following conditions:
    \begin{enumerate}
        \item $u \leq u_r$, $u \leq u_s$;
        \item $H_m(u_r) = \mu_r(u) = {\bf 1}_{\{u > -\infty\}} H_m(u)$;
        \item $H_m(u_s) = \mu_s(u) = {\bf 1}_{\{u = -\infty\}} H_m(u)$.
    \end{enumerate}
Moreover, $u_r + u_s \leq u$.}

\bigskip

The proof of Theorem~A relies on the rooftop techniques developed in Section~\ref{Rooftop envelope}. The most valuable and surprising implication of these techniques is the general comparison principle articulated in Theorem~\ref{thm: uniqueness in N}, which allows us to handle singular Hessian measures satisfactorily for the first time.
Building on this comparison principle, we are able to adapt our argument in the Monge-Amp\`ere case \cite{ACLR24} to obtain the following result.

\bigskip

\noindent {\bf Theorem~B.}\emph{
Assume $H_1, H_2 \in \E$, and
\[
\PP[H_1](H_2) = \PP(g_{H_1},H_2).
\]
Then $\PP[u](v) = \PP(g_u, v)$ for all $u \in \Ker(H_1)$ and $v \in \Ker(H_2)$.
}

\bigskip

We direct the reader to Section~\ref{sect: rofftop equality} for the definition of the rooftop envelope. The theorem under discussion states that the rooftop equality holds for a pair of $m$-subharmonic functions if it is valid for their boundary functions. Specifically, this condition is met for functions in $\Ker$, as the boundary function vanishes.

Rooftop envelopes have recently emerged as a fundamental tool in constructing plurisubharmonic geodesics within geometry. These geodesics are defined as the upper envelopes of sub-geodesics, an idea that builds upon Mabuchi’s seminal work \cite{Mab87} and the findings of Semmes \cite{Sem92} and Donaldson \cite{Don99}. This approach has been further developed and adapted to local contexts by Berman and Berndtsson \cite{BB22arxiv,BB22}, among others. A comprehensive overview of these developments can be found in \cite{Ras23} (see also~\cite{Ras22}).

However, the construction of $m$-subharmonic geodesics when $1 \leq m < n$ presents significant challenges, a longstanding problem that has notably hindered progress in this field, as discussed in \cite{AC23G}. The geodesic equation for plurisubharmonic functions reveal that for $m$-subharmonic functions, one should replace the standard K\"ahler form by its pull back to the product of the domain with the annulus in $\mathbb{C}$. This observation leads us to introduce the concept of $m$-subharmonic functions with respect to semipositive $(1,1)$-forms. Using this concept we have successfully overcome the formidable difficulties encountered in \cite{AC23G}, thereby allowing us to extend our previous results \cite{ACLR24} to the Hessian setting.

\bigskip

\noindent {\bf Theorem~C.}\emph{
 Assume $H_0,H_1\in \mathcal{E}_m$ are connectable by an $m$-subharmonic geodesic segment. Let $u_0\in \Ker(H_0)$, $u_1\in \Ker(H_1)$. Then $u_0$ and $u_1$ are connectable by an $m$-subharmonic geodesic if and only if
 \[
 g_{u_0}\geq u_1\quad \text{and}\quad g_{u_1}\geq u_0.
\]
  In particular, if $g_{H_0}=g_{H_1}$, then the above condition is equivalent to $g_{u_0}=g_{u_1}$.
}

\bigskip

An essential ingredient in the proof of Theorem~C is the following minimum principle for $m$-subharmonic functions, which extends the renowned result by Kiselman in the case when $n = m$ \cite{Kis78}.

\medskip

\noindent {\bf Theorem~D.} \emph{Let $(u_t)_{t\in (0,1)}$ be an $m$-subharmonic sub-geodesic in $\Omega$. Then the function $\inf_{t\in (0,1)} u_t$ is $m$-subharmonic in $\Omega$.}

\bigskip

The infimum of a family of $m$-subharmonic functions is often not $m$-subharmonic (not even upper semicontinuous), as it already fails for plurisubharmonic functions. A crucial property of $m$-subharmonic sub-geodesics is that it is $S^1$-invariant in the annulus variable.

The paper is organized as follows. Section~\ref{Sec:Preliminaries} provides the necessary definitions related to Hessian operators and the Cegrell classes. In Section~\ref{Sec: decomposition of measures}, we discuss key aspects of the decomposition of complex Hessian measures, which are critical for the discussions in Sections~\ref{Rooftop envelope} and~\ref{sect: uniqueness}. The concept of the rooftop envelope is introduced in Section~\ref{Rooftop envelope}. Subsequently, in Section~\ref{sect: uniqueness}, we employ our rooftop techniques to establish the comparison principle (Theorem~\ref{thm: uniqueness in N}) that effectively manages singular Hessian measures. This forms the basis for presenting the proofs of Theorem~A and Theorem~B. We conclude this paper by introducing the concept of $m$-subharmonic functions with respect to semipositive $(1,1)$-forms and utilize this to establish the geodesic connectivity criterion, proving Theorem~D and Theorem~C in Section~\ref{sect: geodesic connectivity}.

\begin{ackn}
The third named author acknowledges partial support from the PARAPLUI ANR-20-CE40-0019 project and the Centre Henri Lebesgue ANR-11-LABX-0020-01. The fourth named author is grateful to Jagiellonian University and Université d'Angers for their support. We are deeply grateful to S{\l}awomir Dinew for his invaluable discussions and insightful comments on the first version of this paper, particularly regarding Theorem~D.
\end{ackn}

\section{Hessian operators and definitions of the Cegrell classes}\label{Sec:Preliminaries}

This section introduces the necessary definitions and summarizes some basic facts. For additional details, we refer the reader to~\cite{AC22Q,AS12,AS13,Ngu13}.

We begin by defining $m$-subharmonic functions and the complex Hessian operator. Consider a bounded domain $\Omega \subset \C^n$, where $n \geq 2$, and let $1 \leq m \leq n$. Define $\mathbb{C}_{(1,1)}$ as the set of $(1,1)$-forms with constant coefficients. We then define the set
\[
\Gamma_m = \left\{ \alpha \in \mathbb{C}_{(1,1)} : \alpha \wedge (dd^c|z|^2)^{n-1} \geq 0, \dots, \alpha^m \wedge (dd^c|z|^2)^{n-m} \geq 0 \right\}.
\]

The real counterpart of $m$-subharmonic functions was first introduced by Caffarelli, Nirenberg, and Spruck~\cite{CNS85}. The origin of these functions in the complex setting, which is our focus here, was established by Vinacua~\cite{Vin86, Vin88}. Later, B\l ocki~\cite{Blo05} extended this concept to unbounded functions, as seen in Definition~\ref{m-sh}, and he introduced pluripotential methods.

\begin{definition}\label{m-sh}
Let $n \geq 2$ and $1 \leq m \leq n$. Assume that $\Omega \subset \C^n$ is a bounded domain, and let $u$ be a subharmonic function defined on $\Omega$. We say that $u$ is \emph{$m$-subharmonic} if it satisfies the following inequality
\[
dd^c u \wedge \alpha_1 \wedge \dots \wedge \alpha_{m-1} \wedge (dd^c |z|^2)^{n-m} \geq 0,
\]
in the sense of currents, for all $\alpha_1, \ldots, \alpha_{m-1} \in \Gamma_m$. The set of all $m$-subharmonic functions defined on $\Omega$ is denoted by $\mathcal{SH}_m(\Omega)$.
\end{definition}

If $u$ is bounded and $m$-subharmonic, then the current $dd^c u \wedge (dd^c |z|^2)^{n-m}$ is positive, and therefore, it has measure coefficients. Following Bedford and Taylor \cite{BT76}, one can inductively define
\[
dd^c u_1 \wedge \ldots \wedge dd^c u_m \wedge (dd^c |z|^2)^{n-m},
\]
as a positive Radon measure in $\Omega$.

The properties of general $m$-subharmonic functions significantly diverge from those of $n$-subharmonic functions. A primary concern is their integrability; whereas all plurisubharmonic functions are locally $L^p$ integrable for any $p > 0$, $m$-subharmonic functions do not necessarily share this property. B\l ocki has conjectured that $m$-subharmonic functions should be locally $L^p$ integrable for $p < \frac{nm}{n-m}$, a conjecture that has received partial confirmation in~\cite{AC20,DK14}.

From Definition~\ref{m-sh} it follows
\[
\mathcal{PSH}=\mathcal{SH}_n \subset \cdots \subset \mathcal{SH}_1 =\mathcal{SH}\, .
\]

\begin{definition}\label{prel_hcx}
Let $n \geq 2$, and $1 \leq m \leq n$. A bounded domain $\Omega \subset \C^n$ is said to be \emph{$m$-hyperconvex} if it admits a non-negative, continuous, and $m$-subharmonic exhaustion function, i.e., there exists an $m$-subharmonic function $\varphi : \Omega \to (-\infty, 0)$ such that the closure of the set $\{z \in \Omega : \varphi(z) < c\}$ is compact in $\Omega$ for every $c \in (-\infty, 0)$.
\end{definition}

For example Hartogs' triangle is $1$-hyperconvex, but not $2$-hyperconvex. For further information on  $m$-hyperconvex domains see~\cite{ACH18}. Next, we shall recall the function classes that are of our interest.

A function $\varphi$, which is $m$-subharmonic on an $m$-hyperconvex domain $\Omega$, belongs to the class $\mathcal{E}^0_{m}(\Omega)$ if $\varphi$ is bounded, satisfies
\[
\lim_{z \to \xi} \varphi(z) = 0 \quad \text{for every } \xi \in \partial\Omega,
\]
and fulfills
\[
\int_{\Omega} \operatorname{H}_m(\varphi) < +\infty.
\]

\begin{definition}
Let $n \geq 2$, and $1 \leq m \leq n$. Assume that $\Omega$ is a bounded $m$-hyperconvex domain in $\mathbb{C}^n$.
We say that $u \in \mathcal{F}_m(\Omega)$ if $u$ is an $m$-subharmonic function defined on $\Omega$ and there exists a decreasing sequence $\{\varphi_j\}$, where each $\varphi_j \in \mathcal{E}^0_m(\Omega)$, such that $\varphi_j$ converges pointwise to $u$ on $\Omega$ as $j \to +\infty$, and $\sup_j \int_{\Omega} \operatorname{H}_m(\varphi_j) < +\infty$. Furthermore, if for every $z \in \Omega$ there exists a neighborhood $V \subset \Omega$ of $z$ and $u_V \in \mathcal{F}_m(\Omega)$ such that $u_V = u$ on $V$, then we say that $u \in \mathcal{E}_m$.
\end{definition}

In~\cite{Lu12,Lu15}, it was proved that for $u\in \mathcal{E}$, the complex Hessian operator, $\operatorname{H}_m(u)$, is well-defined as
\[
\operatorname{H}_m(u) = (dd^c u)^m \wedge (dd^c |z|^2)^{n-m},
\]
where $d = \partial + \bar{\partial}$, and $d^c = \sqrt{-1} (\bar{\partial} - \partial)$.

Let $[\Omega_j]$ denote a \emph{fundamental sequence} in the sense that it is an increasing sequence of $m$-hyperconvex subsets of the bounded $m$-hyperconvex domain $\Omega$, such that $\Omega_j \Subset \Omega_{j+1}$ for every $j \in \mathbb{N}$, and $\bigcup_{j=1}^{\infty} \Omega_j = \Omega$. Then, if $u \in \mathcal{E}$, and $[\Omega_j]$ is a fundamental sequence, we define
\[
u^j = \sup \left\{ \varphi \in \mathcal{SH}_{m}(\Omega) : \varphi \leq u \text{ on } \bar{\Omega}_j \right\},
\]
and
\[
\tilde{u} = \left(\lim_{j \to +\infty} u^j\right)^*,
\]
where $(\,\cdot\,)^*$ denotes the upper-semicontinuous regularization.
The function $\tilde{u}$ is the \emph{smallest maximal $m$-subharmonic majorant} of $u$. Set
\[
\Ker = \left\{u \in \mathcal{E} : \tilde{u} = 0 \right\}.
\]
For Cegrell's fundamental work on these classes for $m=n$, see~\cite{Ceg98,Ceg04,Ceg08}.

\medskip

Next, we introduce the Cegrell classes with generalized boundary values.

\begin{definition}\label{Prel: CC}
Let $\mathcal{K} \in \{\Eo, \F, \Ker, \E\}$. A plurisubharmonic function $u$ on $\Omega$ belongs to the class $\mathcal{K}(\Omega, H) (= \mathcal{K}(H))$, $H \in \mathcal {SH}^-(\Omega)$, if there exists a function $\varphi \in \mathcal{K}$ such that
\[
H \geq u \geq \varphi + H.
\]
\end{definition}
Further information about $\Ker(H)$ can be found in~\cite{EG21, NVT19}. For any subset $\mathcal{K} \subseteq \mathcal{E}(\Omega)$, we introduce the notation
\[
\mathcal{K}^a = \{\varphi \in \mathcal{K} : (dd^c \varphi)^n \text{ vanishes on all $m$-polar sets in } \Omega\}.
\]
On several occasions, we shall use the following maximum principle.

\begin{lemma}\label{HVP17}
\begin{enumerate}
\item Let $u_1,\dots,u_{m-1}\in \E$ and $v\in \mathcal {SH}^-$. Then
\begin{multline*}
dd^c\max(u,v)\wedge dd^cu_1\wedge\cdots\wedge dd^cu_{m-1}\wedge (dd^c|z|^2)^{n-m}|_{\{u>v\}}\\
=dd^cu\wedge dd^cu_1\wedge\cdots\wedge dd^cu_{m-1}\wedge (dd^c|z|^2)^{n-m}|_{\{u>v\}}.
\end{multline*}
In particular
\[
{\bf 1}_{\{u>v\}}\HH(\max(u,v))={\bf 1}_{\{u>v\}}\HH(u).
\]
\item Let $u,v\in \E$ be such that $\HH(u)(\{u=v=-\infty\})=0$, then
\[
\HH(\max(u,v))\geq{\bf 1}_{\{u\geq v\}}\HH(u)+{\bf 1}_{\{u<v\}}\HH(v).
\]
\end{enumerate}
\end{lemma}
\begin{proof}
Part (1) was proved in~\cite{HVP17}.

To prove (2), let us define the following sets for $t>0$, $K_t=\{u+t=v\}\setminus \{u=v=-\infty\}$. Note that $K_{t_1}\neq K_{t_2}$, for $t_1\neq t_2$, so there exists a sequence $t_j\to 0$, such that $\HH(u)(K_{t_j})=0$, and therefore $\HH(u)(\{u+t_j=v\})=0$. Now we can apply point (1) to get
\begin{multline}\label{emp}
    \HH(\max(u,v-t_j)) \\
    \geq {\bf 1}_{\{u> v-t_j\}}\HH(\max(u,v-t_j))+{\bf 1}_{\{u< v-t_j\}}\HH(\max(u,v-t_j))\\
    ={\bf 1}_{\{u\geq v-t_j\}}\HH(u)+{\bf 1}_{\{u< v-t_j\}}\HH(v) \\ \geq {\bf 1}_{\{u\geq v\}}\HH(u)+{\bf 1}_{\{u< v-t_j\}}\HH(v).
\end{multline}
Observe that $\max(u,v-t_j)$ and ${\bf 1}_{\{u< v-t_j\}}$ are increasing sequences, therefore passing with $j\to +\infty$ in (\ref{emp}) we prove (2).

\end{proof}

\section{Cegrell decomposition of measures}\label{Sec: decomposition of measures}

Here, we discuss some aspects of the decomposition of complex Hessian measures.

First, recall the Cegrell-Lebesgue decomposition theorem~\cite{Lu15}: A non-negative Radon measure $\mu$ can be decomposed into a regular (non $m$-polar) and a singular ($m$-polar) part,
\[
\mu = \mu_r + \mu_s,
\]
such that
\[
\mu_r = f \operatorname{H}_m(\varphi),
\]
where $\varphi \in \mathcal{E}_m^0$, $f \geq 0$, and $f \in L^1_{\text{loc}}(\operatorname{H}_m(\varphi))$. The measure $\mu_s$ is carried by an $m$-polar subset of $\Omega$. Furthermore, if $\mu$ is a Hessian measure of some $u \in \mathcal{E}_m$, i.e., $\mu = \operatorname{H}_m(u)$, then $\mu_s$ is carried by $\{z \in \Omega : u(z) = -\infty\}$ and
\[
\mu_s(u) = {\mathbf{1}}_{\{u = -\infty\}}(dd^c u)^n \quad \text{and} \quad \mu_r(u) = {\mathbf{1}}_{\{u > -\infty\}}(dd^c u)^n.
\]
We will now demonstrate that the regular part can be described alternatively, using the standard approximation $u_j = \max(u, -j)$ of the $m$-subharmonic function $u$.

\begin{lemma}\label{reg}
Let $u\in \mathcal{SH}_m$, and let $u_j=\max(u,-j)$, $j>0$. Then
\begin{enumerate}
    \item the sequence of measures ${\bf 1}_{\{u>-j\}}\HH(u_j)$ is increasing
    and hence we can define
\[
\mu_r(u)  = \lim_{j\to +\infty} {\bf 1}_{\{ u>-j\}}\HH(u_j);
\]
\item  and for any $k>0$ holds
    \begin{equation}
	\label{eq: NP plurifine property}
	{\bf 1}_{\{u>-k\}} \HH(u_k)  = {\bf 1}_{\{u>-k\}}  \mu_r(u).
\end{equation}

    \item If in addition $u\in \E$, then
\[
\mu_r(u)={\bf 1}_{\{u>-\infty\}}\HH(u).
\]

\end{enumerate}
\end{lemma}
\begin{proof}
To prove the first part of the lemma fix $j>k>0$ and observe that $u_k=\max(u_j,-k)$ and
\[
\{u_j>-k\} = \{u>-k\} \subset \{u>-j\}.
\]
Therefore using Lemma~\ref{HVP17}, we get
\begin{flalign}
{\bf 1}_{\{u>-k\}} \HH(u_k) &= {\bf 1}_{\{u_j>-k\}}   \HH(\max(u_j,-k))=  {\bf 1}_{\{u_j>-k\}}   \HH(u_j)\nonumber \\
&  = {\bf 1}_{\{u>-k\}}   \HH(u_j) \leq {\bf 1}_{\{u>-j\}}   \HH(u_j). \label{eq: NP MA}
\end{flalign}
To prove (2) note that
\[
{\bf 1}_{\{u>-k\}} \HH(u_k)  = {\bf 1}_{\{u>-k\}}   \HH(u_j) = {\bf 1}_{\{u>-k\}} {\bf 1}_{\{u>-j\}}   \HH(u_j),
\]
and to get (\ref{eq: NP plurifine property}) it is enough to put $j\to +\infty$.

Finally we shall prove point (3) which says that the measure obtained in point (1) is the regular part of the Hessian measure in the sense of Cegrell.
It is enough to show that for any Borel set $E\subset \Omega\setminus\{u=-\infty\}$ holds
\[
\lim_{j\to +\infty}\HH(u_j)(E\cap {\{ u>-j\}})=\HH(u)(E).
\]
Without lost of generality we can assume that $u\in\F$. For $j>0$ define $m$-subharmonic function by $v^j=\max(j^{-1}u,-1)+1$ and note that $v^j\nearrow {\bf 1}_{\{u>-\infty\}}$, $j\to +\infty$, and $v^j=0$ on the set $\{u\leq -j\}$. We shall prove that
\begin{equation}\label{1}
v^j\HH(u_j)=v^j\HH(u).
\end{equation}
By~\cite{ACH18} there exists decreasing sequence of continuous functions $u^k\in \Eo$ converging pointwise to $u$, as $k\to +\infty$. Then it follows from~\cite{HVP17} that
\[
\lim_{k\to +\infty}v^j\HH((u^k)_j)=v^j\HH(u_j), \ \ \text{and} \ \  \lim_{k\to +\infty}v^j\HH(u^k)=v^j\HH(u).
\]
Recall that $(u^k)_j:=\max(u^k,-j)$ and we have $(u^k)_j=u^k$ on the open set $\{u^k>-j\}$ containing $\{u>-j\}$, so \eqref{1} holds.

Note that by monotone convergence theorem and by (\ref{eq: NP plurifine property})
\[
v^j\HH(u_j)={\bf 1}_{\{u>-j\}}v^j\HH(u_j)={\bf 1}_{\{u>-j\}}v^j\mu_r(u)\nearrow {\bf 1}_{\{u>-\infty\}}\mu_r(u),
\]
as $j\to +\infty$. On the other hand
\[
v^j\HH(u)\nearrow {\bf 1}_{\{u>-\infty\}}\HH(u), \ \ j\to +\infty.
\]
By (\ref{1}) the proof is finished.
\end{proof}

\begin{remark}
    Note that Lemma~\ref{reg} implies that the regular (or non $m$-polar) part $\mu_r(u)$ of the complex Hessian measure can be defined for any $m$-subharmonic function $u$ as a limit of an increasing sequence of measures ${\bf 1}_{\{u > -j\}}\HH(u_j)$. However, in a general situation, it may not be a regular Borel measure, as it can become unbounded near the set $\{u = -\infty\}$.
\end{remark}

We shall prove the following maximum principle concerning the regular part. A total mass version of Theorem~\ref{thm: maximum principle} is documented in Lemma~\ref{HVP17}.

\begin{theorem}\label{thm: maximum principle}
	If $u, v \in \E$, then
	\[
	 \mu_r(\max(u,v))    \geq {\bf 1}_{\{u\geq v\}} \mu_r(u)  +  {\bf 1}_{\{u < v\}} \mu_r(v).
	\]
	If, in addition, $u\leq v$ then
	\[
	 {\bf 1}_{\{u = v\}} \mu_r(v)  \geq {\bf 1}_{\{u = v\}} \mu_r(u).
	\]
\end{theorem}
\begin{proof}
	Let $u_j= \max(u,-j)$, $v_j=\max(v,-j)$, $j>0$. From the maximal principle Lemma~\ref{HVP17} we get
	\begin{equation}\label{MP}
	\HH(\max(u_j,v_j))   \geq {\bf 1}_{\{u_j\geq v_j\}}   \HH(u_j)  +  {\bf 1}_{\{u_j< v_j\}}  \HH(v_j).
	\end{equation}

Note that for $k<j$ we get
\[
\{\min(u,v)>-k\}\subset \{u>-k\}\subset \{u>-j\}
\]
(and the same inclusions hold for the function $v$) and on the set $\{\min(u,v)>-k\}$ holds $u_j=u$ and $v_j=v$. Therefore multiplying (\ref{MP}) with ${\bf 1}_{\{\min(u,v)>-k\}}$, and using~\eqref{eq: NP plurifine property}, we obtain
	\begin{flalign*}
			{\bf 1}_{\{\min(u,v)>-k\}} \mu_r(\max(u,v))   & \geq {\bf 1}_{\{\min(u,v)>-k\}}  {\bf 1}_{\{u\geq v\}}   \mu_r(u)   \\
			&+ {\bf 1}_{\{\min(u,v)>-k\}}  {\bf 1}_{\{u< v\}}  \mu_r(v).
	\end{flalign*}
To finish the proof it is enough to let $k\to +\infty$, since regular parts of Hessian measures vanish on $m$-polar sets.
\end{proof}

We next recall the maximum principle for the singular parts, see~\cite{HVP17}.

\begin{theorem}\label{thm: maximum principle SP}
	If $u, v \in \E$, $u\leq v$ then
	\[
	 \mu_s(v)  \leq \mu_r(u).
	\]
\end{theorem}

We shall end this section with the following remark concerning the weak convergence of regular and singular parts of the Hessian measure.

\begin{remark}\label{thm: lsc of NP}
It is well known that monotone convergence of $m$-subharmonic functions implies weak convergence of corresponding Hessian measures. Now the question is: Does similar convergence hold for regular and singular parts of Hessian measures? The answer is quite surprising. Arguing as in~\cite{ACLR24},  one can prove that for increasing sequences $u_j\nearrow u$, one has:
    \[
    \mu_r(u_j) \to \mu_r(u)\quad  \text{and} \quad  \mu_s(u_j) \to \mu_s(u).
    \]
Note that a similar result does not hold for decreasing sequences, as indicated in the following example of plurisubharmonic functions $u_j(z)=\max(\log |z|,-j)\searrow u(z)=\log |z|$ in the unit ball in $\C^n$. Instead, we have only that
    \[
    \liminf_{j\to +\infty}\mu_r(u_j)\geq \mu_r(u) \quad  \text{and} \quad \limsup_{j\to +\infty}\mu_s(u_j)\leq \mu_s(u).
    \]
\end{remark}

\section{Rooftop envelopes}\label{Rooftop envelope}

Envelope constructions have been central to classical potential theory since the foundational work by Oskar Perron~\cite{Perron1923}. This line of research was further advanced by Wiener in his series of articles~\cite{Wiener24a,Wiener24b, Wiener25} and was later expanded by Brelot~\cite{Brelot1939}. Additionally, the concept of envelope constructions has proved to be even more crucial in pluripotential theory since 1959, when Bremermann adapted the methodologies of Perron and Carath{\'e}odory to the pluricomplex setting~\cite{Bremermann1959,CC37}. The publication of~\cite{BT76,BT82} established envelope constructions as a fundamental tool in pluripotential theory.

In our study of $m$-subharmonic functions, the perspectives of both potential and pluripotential theories are relevant. As Definition~\ref{m-sh} implies there is an inclusion hierarchy among these function spaces:
\[
\mathcal{PSH}=\mathcal{SH}_n \subset \cdots \subset \mathcal{SH}_1 =\mathcal{SH}\, .
\]
Here, $\mathcal{PSH}$ is the space of plurisubharmonic functions, and $\mathcal{SH}$ the space of subharmonic functions.

Typically, the envelope $\PP(h)$, defined later, is considered for functions that are either continuous or lower semi-continuous. However, in our approach, we use the function $h$ either as the minimum or the difference between two $m$-subharmonic functions. Consequently, we introduce the following definition for the rooftop envelope.

\begin{definition}
    For a function $h : \Omega \to \mathbb{R} \cup \{-\infty\}$, which is bounded from above, we define the envelope $\PP(h)$ as
    \[
    \PP(h)(z) = \left( \sup \{ u(z) : u \in \mathcal{SH}_m(\Omega), u \leq h \text{ quasi-everywhere in } \Omega \}\right)^*,
    \]
    with the convention that $\sup \emptyset = -\infty$. If no $m$-subharmonic function is below $h$ quasi-everywhere, then $\PP(h)$ is defined to be identically $-\infty$.
\end{definition}
Here, quasi-everywhere means outside an $m$-polar set. Recall that a subset $E \subset \Omega$ is $m$-polar if for every $z \in \Omega$, there exists a neighborhood $U$ of $z$ and a function $u \in \mathcal{SH}_m(U)$ such that $E \cap U \subseteq \{u = -\infty\}$.

We shall also consider the Hessian capacity:
\[
\operatorname{Cap}_{m}(E) = \sup \left\{ \int_E \operatorname{H}_m(u) : u \in \mathcal{SH}_m(\Omega), -1 \leq u \leq 0 \right\}.
\]

A function $h$ is said to be \emph{quasi-continuous} if, for each $\varepsilon > 0$, there exists an open set $U$ such that $\operatorname{Cap}_{m}(U, \Omega) < \varepsilon$, and $h$ is continuous on $\Omega \setminus U$. Similarly, a set $E$ is quasi-open if, for each $\varepsilon > 0$, there exists an open set $U$ such that
\[
\operatorname{Cap}_{m}(U \setminus E \cup E \setminus U) \leq \varepsilon.
\]
If a function $u \in \mathcal{SH}_m(\Omega)$ is below $h$ quasi-everywhere, then $\PP(h) \in \mathcal{SH}_m(\Omega)$, and Choquet's lemma implies that there exists an increasing sequence $[u_j] \subset \mathcal{SH}_m(\Omega)$ such that $u_j \leq h$ quasi-everywhere in $\Omega$ and $(\lim_{j \to +\infty} u_j)^* = \PP(h)$. Furthermore, the set
\[
\{z \in \Omega : \lim_{j \to +\infty} u_j(z) < \PP(h)(z)\}
\]
is $m$-polar, and $\PP(h) \leq h$ quasi-everywhere in $\Omega$.

\begin{theorem}\label{thm: envelope}
	If $h$ is quasi-continuous in $\Omega$ and $\PP(h) \not \equiv -\infty$, then
	\[
	\int_{\{\PP(h) < h\}}  \mu_r (\PP(h))   = 0.
	\]
\end{theorem}
\begin{proof}
Note that there exists a sequence $[h_j]$ of bounded from below and lower semi-continuous functions in $\Omega$ such that $h_j \searrow h$ quasi-everywhere in $\Omega$. Now we can adapt well known balayage procedure to conclude that (cf.~\cite[Lemma~4.9]{Lu15}),
	\[
	\int_{\{\PP(h_j) < h_j\}} \HH(\PP(h_j))=0, \quad \text{for all } j\in \mathbb N.
	\]
	Now fix $k<j$, and note that $\{\PP(h_k)<h\} \subset \{\PP(h_j)<h_j\}$, so
	\[
	\int_{\{\PP(h_k) < h\}} \HH(P(h_j)) =0.
	\]
	Let $t>0$, we can use Lemma~\ref{HVP17} to get
	\[
 \begin{aligned}
	&\int_{\{\PP(h_k) < h\}} {\bf 1}_{\{\PP(h)>-t\}} \HH(\max(\PP(h_j),-t))\\
 &\leq \int_{\{\PP(h_k) < h\}} {\bf 1}_{\{\PP(h_j)>-t\}} \HH(\max(\PP(h_j),-t)) \\
 &=\int_{\{\PP(h_k) < h\}} {\bf 1}_{\{\PP(h_j)>-t\}} \HH(\PP(h_j))=0.
\end{aligned}
 \]
The set $\{\PP(h_k) < h\}\cap\{\PP(h)>-t\}$ is quasi-open and the sequence $\max(\PP(h_j),-t)$ is uniformly bounded, therefore corresponding sequence of the Hessian measures is uniformly bounded by $\Capo$ and then 	
\[
 \begin{aligned}
    & 0=\liminf_{j\to +\infty}\int_{\{\PP(h_k) < h, \PP(h)>-t\}} \HH(\max(\PP(h_j),-t))\\
    &\geq \int_{\{\PP(h_k) < h\}}{\bf 1}_{\{\PP(h)>-t\}} \HH(\max(\PP(h),-t)).
 \end{aligned}
	\]
To finish the proof it is enough to let $k\to +\infty$, and then $t\to +\infty$.
\end{proof}

Now, we shall introduce the definition of the rooftop envelope of $m$-subharmonic functions, originated from \cite{Dar15}, \cite{DR16}, \cite{Dar17AJM}. The \emph{rooftop envelope} $\PP(u,v)$, for any two $m$-subharmonic functions $u$ and $v$, is defined as the $m$-subharmonic envelope of $\min(u,v)$. This represents the largest $m$-subharmonic function that lies below $\min(u,v)$, e.g., $\PP(u,v) = \PP(\min(u,v))$. The rooftop envelope is always well-posed for $u,v \in \mathcal{E}_m$. Since $u+v \leq \min(u,v)$, it follows that $u+v \leq \PP(u,v)$ and therefore $\PP(u,v) \in \mathcal{E}_m$.

We now describe the behavior of the Hessian measure of the rooftop envelopes.

\begin{theorem}\label{cor: NP of rooftop}
Let $u,v\in \E$ then
    \[
    \mu_r(\PP(u,v))\leq {\bf 1}_{\{\PP(u,v)=u\}}\mu_r(u) + {\bf 1}_{\{\PP(u,v)=v, \PP(u,v)<u\}} \mu_r(v).
    \]
In particular, if $\mu$ is a positive measure such that $\mu_r(u)\leq \mu$ and $\mu_r(v)\leq \mu$, then $\mu_r(\PP(u,v))\leq \mu$.
\end{theorem}
\begin{proof}
    It follows from Theorem~\ref{thm: envelope}, that $\mu_r(\PP(u,v))$ is supported on the contact set $K=K_u\cup K_v$, where
    \[
    K_u=\{\PP(u,v)=u\}\quad \text{and}\quad K_v=\{\PP(u,v)=v\}\cap \{\PP(u,v)<u\}.
    \]
    Theorem \ref{thm: maximum principle} then yields
\[
\begin{aligned}
    {\bf 1}_{K_u} \mu_r(\PP(u,v)) &\leq {\bf 1}_{K_u} \mu_r(u), \\
    {\bf 1}_{K_v} \mu_r(\PP(u,v)) &\leq {\bf 1}_{K_v} \mu_r(v),
\end{aligned}
\]
which finish the proof.
\end{proof}

The above theorem was proved for $m$-subharmonic functions from the Cegrell class with finite energy in~\cite{AC23G}.

\section{Decomposition of $m$-subharmonic functions} \label{sect: uniqueness}

The objective of this section is to establish a novel decomposition theorem for $m$-subharmonic functions within the class $\Ker$, as outlined in the introduction (Theorem~A). As a consequence of the rooftop techniques developed in Section~\ref{Rooftop envelope}, we establish in Theorem~\ref{thm: uniqueness in N} a general comparison principle in the Cegrell class $\Ker(H)$. This principle will serve as the foundational tool in Theorem~A, and later also in Theorem~B as well as in Theorem~C.

\subsection{Comparison principle}

We begin by defining two relations, $\preceq$ and $\simeq$, among $m$-subharmonic functions. These relations relate to the singularities of $m$-subharmonic functions.

\begin{definition}\label{def: singular}
  For $m$-subharmonic functions $u$ and $v$ defined in $\Omega$, we say $u$ \emph{is more singular than} $v$ if, for any compact subset $K \Subset \Omega$, there exists a constant $C_K$ such that $u \leq v + C_K$ throughout $K$. This relation is denoted by $u \preceq v$. If $u$ is more singular than $v$ and $v$ is more singular than $u$, we say that $u$ and $v$ share \emph{identical singularities}, we denote this by $u \simeq v$.
\end{definition}

We proceed with the following lemma, whose counterpart for plurisubharmonic functions was proved in~\cite{ACCP09}. It asserts that two $m$-subharmonic functions sharing identical singularities also share identical singular Hessian measures.

\begin{lemma}\label{lem: ACCP09 lemma 4.1}
If $u,v\in \E$ and $w\in \mathcal E_m^a$ are such that $|u-v|\leq -w$, then $\mu_s(u)=\mu_s(v)$.
In particular if $u \simeq v$, then $\mu_s(u) = \mu_s(v)$.
\end{lemma}
\begin{proof}
    The proof is the direct adaptation for $m$-subharmonic function of the proof of Lemma 4.12 in~\cite{ACCP09}.
\end{proof}

The converse result is not true, two $m$-subharmonic functions with different type of singularities can have the same singular Hessian measure. However if we assume that $u \preceq v$ then it turns out that the difference $u-v$ is not very singular.

\begin{lemma}\label{prop: P(u-v) is in Ea}
    Let $u,v \in \E$ and $u \preceq v$.
    \begin{enumerate}
    \item If in addition ${\bf 1}_{\{u = -\infty\}}\HH(u) = {\bf 1}_{\{v = -\infty\}}\HH(v)$,
    then $\HH(\PP(u-v,0)) \leq \mu_r(u)$. In particular, $\PP(u-v,0) \in \mathcal E_{m}^a$.
    \item If in addition $\HH(u) \leq \HH(v)$, then $\HH(P(u-v,0)) = 0$.
\end{enumerate}
\end{lemma}
\begin{proof}
\emph{Part (1).} Define $\varphi_j = \max(u, v - j)$ for any $j > 0$. The function $\varphi_j$ belongs to $\E$ and $\varphi_j \simeq v$. By Lemma~\ref{lem: ACCP09 lemma 4.1}, we have:
\[
{\bf 1}_{\{\varphi_j = -\infty\}} \HH(\varphi_j) = {\bf 1}_{\{v = -\infty\}} \HH(v) = {\bf 1}_{\{u = -\infty\}} \HH(u).
\]
As $\varphi_j \searrow u$, it follows from~\cite{Lu15} that $\HH(\varphi_j)$ converges weakly to $\HH(u)$. Consequently, by the Cegrell decomposition:
\[
\HH(\varphi_j) = \mu_r(\varphi_j) + {\bf 1}_{\{u = -\infty\}} \HH(u) \rightarrow \mu_r(u) + {\bf 1}_{\{u = -\infty\}} \HH(u),
\]
implying that $\mu_r(\varphi_j)$ converges weakly to $\mu_r(u)$.

Next, define $w_j = \PP(\varphi_j - v, 0)$, observing that $w_j$ belongs to $\E \cap L^{\infty}(\Omega)$ and satisfies $w_j + v \leq \varphi_j$, with equality on the contact set $K_j = \{w_j + v = \varphi_j\}$. By Theorem~\ref{thm: maximum principle} and Theorem~\ref{thm: envelope}, we obtain:
\[
\HH(w_j) = {\bf 1}_{K_j} \HH(w_j) \leq {\bf 1}_{K_j} \mu_r(w_j + v) \leq {\bf 1}_{K_j} \mu_r(\varphi_j) \leq \mu_r(\varphi_j).
\]
Since $w_j \searrow w = \PP(u-v, 0) \in \E$, and given the weak convergence $\mu_r(\varphi_j) \to \mu_r(u)$ from~\cite{HVP17}, it follows that $\HH(w) \leq \mu_r(u)$, ensuring $w \in \mathcal E_{m}^a$.

\emph{Part (2).} Given that $\HH(u) \leq \HH(v)$ and $u \preceq v$, we derive from~\cite{HVP17} the following:
\begin{multline*}
{\bf 1}_{\{v=-\infty\}} \HH(v) \leq {\bf 1}_{\{u=-\infty\}} \HH(u) \leq {\bf 1}_{\{u=-\infty\}} \HH(v) \\
= {\bf 1}_{\{u=-\infty\}} ({\bf 1}_{\{v=-\infty\}} \mu_s(v) + \mu_r(v)) \\
= {\bf 1}_{\{u=-\infty\}} {\bf 1}_{\{v=-\infty\}} \mu_s(v) = {\bf 1}_{\{v=-\infty\}} \HH(v).
\end{multline*}
The final equality follows because $\mu_r(v)$ does not put mass on the $m$-polar set $\{u=-\infty\}$. Therefore, we have
\[
{\bf 1}_{\{u=-\infty\}} \HH(u) = {\bf 1}_{\{v=-\infty\}} \HH(v),
\]
implying that
\[
\mu_r(u) \leq \mu_r(v).
\]
Define $w = \PP(u-v, 0)$ and defining  $K = \{w + v = u\}$. From Part (1), we know $w \in \mathcal{E}_{m}^a$. Given that $w + v \leq u$ with equality on $K$, application of Theorem~\ref{thm: maximum principle} yields
\[
{\bf 1}_K \mu_r(w) + {\bf 1}_K \mu_r(v) \leq {\bf 1}_K \mu_r(u) \leq {\bf 1}_K \mu_r(v),
\]
leading to ${\bf 1}_K \mu_r(w) = 0$. Moreover, according to Theorem~\ref{thm: envelope}, $\mu_r(w)$ is supported on $K$, hence $\mu_r(w) = 0$. This completes the proof.
\end{proof}

Before establishing the general comparison principle, let us start by examining a special case.

\begin{proposition}\label{prop: uniqueness MA zero}
    If $u\in \Ker$ and $(dd^c u)^n=0$, then $u=0$.
\end{proposition}
\begin{proof}
    This is a particular case of a result established in \cite{NVT19}, where the author adapted arguments used in the Monge-Amp\`ere scenario. Below, we present an alternative proof employing the envelope technique.

    Initially, we assume that $u \in \mathcal{F}_m$. We fix $\phi \in \Eo$ and define $u_t = \max(u, t\phi)$ for each $t > 0$. According to \cite[Theorem 3.22]{Lu15}, we have $\int_{\Omega} \HH(u_t) = 0$. Since $u_t \in \Eo$, the comparison principle implies $u_t = 0$, and consequently, $u = 0$.

    To address the general case, we consider a fundamental sequence $[\Omega_j]$ of $\Omega$ and define $u^j := \PP({\bf 1}_{K_j} u)$, where $K_j = \Omega \setminus \Omega_j$. By the definition of $\Ker$, $u^j \nearrow 0$. The function $\PP(u - u^j)$ is in $\mathcal{F}_m$ because $\PP(u - u^j) \geq \PP({\bf 1}_{\Omega_j}u) \in \mathcal{F}_m$, and by Lemma~\ref{prop: P(u-v) is in Ea}, we find $\HH(\PP(u - u^j)) = 0$. It therefore follows that $\PP(u - u_j) = 0$, leading to $u \geq u^j$. As this holds for all $j$ and $u^j \nearrow 0$, we conclude that $u = 0$.
\end{proof}

The most cherished implication of our rooftop techniques so far is the following general comparison principle.

\begin{theorem}\label{thm: uniqueness in N}
Let $H \in \mathcal{E}_m$, and let $u, v \in \mathcal{N}_m(H)$ with $u \preceq v$. If $\operatorname{H}_m(u) \leq \operatorname{H}_m(v)$, then $u \geq v$. In particular, if $\operatorname{H}_m(u) = \operatorname{H}_m(v)$, then $u = v$.
\end{theorem}
\begin{proof} Directly from the definition there exists $\varphi \in \mathcal{N}_{m}$ such that $\varphi + H \leq u \leq H$. Then we have $u - v \geq \varphi$ which implies  $\PP(u - v, 0) \in \mathcal{N}_m$. From Lemma~\ref{prop: P(u-v) is in Ea} we get that $\HH(\PP(u-v,0)) = 0$. It follows from Proposition \ref{prop: uniqueness MA zero} that $\PP(u-v,0) = 0$, and therefore $u \geq v$.

The second statement directly follows from the first. If $\HH(u) = \HH(v)$, then $u \geq v$, implying $u \simeq v$. Changing the roles of $u$ and $v$ one gets $v \geq u$, thus concluding $u = v$.
\end{proof}

Theorem~\ref{thm: uniqueness in N} was previously established under the condition that
\[
\int_{\Omega} (-w) \HH(u) < +\infty
\]
for some $w \in \mathcal{E}_m^0$ with $w < 0$; see \cite{NVT19}. However, as demonstrated in \cite[Example 5.3]{Ceg08}, there exists a function $u \in \mathcal{N}_n \cap L^{\infty}$ such that
\[
\int_{\Omega} (-w) \operatorname{H}_{n}(u) = +\infty
\]
for all $w \in \mathcal{SH}_n^-$ with $w < 0$.

\subsection{Decomposition of $m$-subharmonic functions}

Thanks to Theorem~\ref{thm: uniqueness in N}, we are now prepared to present a proof of Theorem~A, as highlighted in the introduction.

\bigskip

\noindent {\bf Theorem~A}. \emph{
For any $u \in \Ker$, there exist unique functions $u_r, u_s \in \Ker$ satisfying the following conditions:
    \begin{enumerate}
        \item $u \leq u_r$, $u \leq u_s$;
        \item $\HH(u_r) = \mu_r(u) = {\bf 1}_{\{u > -\infty\}}\HH(u)$;
        \item $\HH(u_s) = \mu_s(u) = {\bf 1}_{\{u = -\infty\}}\HH(u)$.
    \end{enumerate}
Moreover, $u_r + u_s \leq u$.}
\begin{proof}
It follows from  Theorem 6.3 (2) in~\cite{NVT19} that there exists $u_r, u_s\in \E$ satisfying the three conditions in the theorem. By Lemma~\ref{prop: P(u-v) is in Ea} we get that $\PP(u-u_s,0)\in \mathcal E_m^a$ and
\[
\HH(\PP(u-u_s,0))\leq \HH(u_r).
\]
By~\cite[Corollary 5.8]{NVT19} we obtain the uniqueness of $u_r$ and we also have $u_r \leq \PP(u-u_s,0)$, hence $u_r+u_s\leq u$.

Now we prove the uniqueness of $u_s$.
Assume now that $v\in \Ker$ is such that $u\leq v$ and $\HH(v) = \HH(u_s)={\bf 1}_{\{u=-\infty\}}\HH(u)$. Then $w=\PP(u_s,v)\in \Ker$ and, by~\cite{HVP17}, since $u\leq w\leq \min(u_s,v)$,
we get
\[
{\bf 1}_{\{w=-\infty\}}\HH(w) = {\bf 1}_{\{v=-\infty\}} \HH(v) = {\bf 1}_{\{u_s=-\infty\}} \HH(u_s).
\]
Therefore we obtain $\HH(w)=\HH(v)=\HH(u_s)$ because these measures are supported by $m$-polar sets.
Theorem \ref{thm: uniqueness in N} then ensures that $v=w=u_s$. This ends the proof.
\end{proof}

\subsection{Rooftop equality} \label{sect: rofftop equality}

\begin{definition}\label{def: residual function}
For $u,v \in \mathcal{SH}_m^-(\Omega)$, the \emph{asymptotic rooftop envelope} $\PP[u](v)$ is defined as follows
\[
\PP[u](v) = \left( \lim_{C \to +\infty} \PP(u + C, v) \right)^*.
\]
In the case where $v = 0$, we denote $g_u = \PP[u](0)$ and call it the \emph{Green-Poisson residual function} of $u$, or simply the \emph{residual function} of $u$.
\end{definition}

The condition $g_u = 0$ means that the $m$-subharmonic function $u$ lacks strong singularities within the domain $\Omega$ and on its boundary $\partial\Omega$. We say that the \emph{rooftop equality} holds for $m$-subharmonic functions $u, v$, if the following holds:
\begin{equation}\label{eq:RTi}
\PP[u](v) = \PP(g_u,v).
\end{equation}

The main result, Theorem~B, of this section is to prove that the rooftop equality holds  for all $u\in \Ker(H_1)$, $v\in \Ker(H_2)$ if it holds for $H_1,H_2$. In particular, the rooftop equality holds in $\Ker$. We start with following result.

\begin{theorem}\label{thm: MA measure of asymptotic envelope}
	Let $u\in \E$. Then
 \begin{enumerate}
     \item $\mu_r(g_u) =0$;

     \item $\HH(\PP(u-g_u))\leq \mu_r(u)$, so $\PP(u-g_u)\in \mathcal E_m^a$;
     \item $\HH(g_u) = \mu_s(u)$.
 \end{enumerate}
 \end{theorem}
\begin{proof}
\emph{Part (1).} Define $v_j = \PP(u+j, 0)$ for $j > 0$. The sequence $\{v_j\}$ is increasing almost everywhere in $\Omega$, and converges to $g_u$ ($v_j \nearrow g_u$). Consequently, the regular part of the measure, $\mu_r(v_j)$, weakly converges to $\mu_r(g_u)$, as indicated by Remark~\ref{thm: lsc of NP}. From Corollary~\ref{cor: NP of rooftop}, we have
\[
\mu_r(v_j) \leq {\bf 1}_{\{v_j = u + j\}} \mu_r(u) \leq {\bf 1}_{\{u \leq -j\}} \mu_r(u) \to 0
\]
as $j \to +\infty$. This leads to $\mu_r(g_u) = 0$.

\emph{Part (2) and (3).} Define $w_j = \PP(u - v_j) = \PP(u - v_j, 0)$. Each $w_j$ belongs to $\mathcal{SH}_m^- \cap L^\infty(\Omega)$ and decreases to $\PP(u - g_u)$. Utilizing the argument from the proof of Lemma~\ref{prop: P(u-v) is in Ea}, we find that
\[
\HH(w_j) \leq \mu_r(u),
\]
and $\HH(\PP(u - g_u)) \leq \mu_r(u)$. Specifically, $\PP(u - g_u) \in \mathcal{E}_m^a$. Considering
\[
g_u + \PP(u - g_u) \leq u \leq g_u,
\]
and referencing Lemma~\ref{lem: ACCP09 lemma 4.1}, we conclude that $\mu_s(g_u) = \mu_s(u)$.
\end{proof}

\begin{remark}
It can be proved, in a similar manner as in~\cite{ACLR24}, that for any $u \in \mathcal{SH}_m^-$, $\mu_r(g_u) = 0$. See Theorem~\ref{thm: MA measure of asymptotic envelope} point (1).
\end{remark}

The following technical lemma in the case of plurisubharmonic functions was proved in~\cite{ACLR24}.

\begin{lemma}\label{lem: MA of w is zero}
   If $u,v \in \E$, then
    \[
    \HH(\PP(\PP[u](v)-\PP(g_u,v)))=0.
    \]
\end{lemma}

\bigskip

\noindent {\bf Theorem~B.}\emph{
Assume $H_1, H_2 \in \E$,  and
\[
\PP[H_1](H_2) = \PP(g_{H_1},H_2).
\]
Then $\PP[u](v) = \PP(g_u, v)$ for all $u \in \Ker(H_1)$ and $v \in \Ker(H_2)$.
}
\begin{proof}
Define the function
\[
\varphi = \PP\left(\PP[u](v) - \PP(g_u, v)\right),
\]
and observe that since $\PP[u](v) \leq \PP(g_u, v)$, it follows that $\varphi \leq 0$. We aim to demonstrate that $\varphi \in \mathcal{N}_m^a$.

Given the assumptions $u \in \Ker(H_1)$ and $v \in \Ker(H_2)$, we find that $\PP(u - H_1) \in \Ker$ and $\PP(v - H_2) \in \Ker$. Consequently, we have:
\begin{flalign*}
\PP(u + C, v) & \geq \PP(\PP(u - H_1) + H_1 + C, \PP(v - H_2) + H_2) \\
& \geq \PP(u - H_1) + \PP(v - H_2) + \PP(H_1 + C, H_2).
\end{flalign*}
If $C \nearrow +\infty$, this leads to:
\begin{multline*}
\PP[u](v) \geq \PP(u - H_1) + \PP(v - H_2) + \PP[H_1](H_2) \\
= \PP(u - H_1) + \PP(v - H_2) + \PP(g_{H_1}, H_2) \\
\geq \PP(u - H_1) + \PP(v - H_2) + \PP(g_u, v).
\end{multline*}

From the above inequality, we see that $\varphi \geq \PP(u - H_1) + \PP(v - H_2)$, so $\varphi \in \Ker$. We now prove that the Hessian measure of $\varphi$ does not put mass on $m$-polar sets. Noting that
\[
\PP[u](v) \geq \PP(u, v) \geq \PP(u - g_u) + \PP(g_u, v),
\]
we deduce that $\varphi \geq \PP(u - g_u)$. It follows from Theorem~\ref{thm: MA measure of asymptotic envelope} that $\PP(u - g_u) \in \mathcal{E}_m^a$, and thus $\varphi \in \mathcal{N}_m^a$. By Lemma~\ref{lem: MA of w is zero}, we have $\HH(\varphi) = 0$, and according to Theorem~\ref{thm: uniqueness in N}, $\varphi = 0$. This concludes the proof.
\end{proof}

Let us note that the equality $\PP[H_1](H_2) = \PP(g_{H_1}, H_2)$ in Theorem~B holds if there exists $w \in \E$ such that $H_2 + w \leq H_1$ and $g_w = 0$. We then observe the following:
\[
\PP[H_1](H_2) \geq \PP[H_2 + w](H_2) \geq H_2 + g_w = H_2  \geq \PP(g_{H_1}, H_2) \geq \PP[H_1](H_2).
\]
This confirms the rooftop equality under the specified conditions.

\section{Geodesic connectivity} \label{sect: geodesic connectivity}

Fix $\Omega$, an open set in $\mathbb{C}^n$, $n\geq 2$, and let $A = \{w \in \mathbb{C} : 1 < |w| < e\}$ denote the annulus in $\mathbb{C}$ with radii $1$ and $e$. Define $\pi : D = \Omega \times A \rightarrow \Omega$ to be the projection mapping $(z,w) \in \Omega \times A$ to $z$ in $\Omega$. Set $\theta = \pi^*(dd^c |z|^2)$, and $\omega=dd^c (|z|^2+|w|^2)$.

\subsection{$m$-subharmonic functions with respect to a semipositive form}
\begin{definition}
A $(1,1)$-form $\alpha$ is $(\theta, m+1)$-positive in $D$ if
\begin{equation*}
    \alpha^k \wedge \omega^{m+1-k}\wedge \theta^{n-m} \geq 0, \quad  \text{for all } 1\leq k\leq m+1,
\end{equation*}
pointwise in $D$.
\end{definition}
In particular, any semipositive $(1,1)$-form is $(\theta,m+1)$-positive. From the definition, it follows that any $(\theta,m+1)$-positive form $\alpha$ satisfies
\[
(\alpha+t\omega)^{m+1}\wedge \theta^{n-m} >0, \; t>0.
\]
This observation is useful in the following version of G{\aa}rding's inequality.
    \begin{lemma}\label{GC: Lem}
    A $(1,1)$-form $\beta$ is $(\theta,m+1)$-positive if and only if
    \begin{equation}
         \label{eq: hyperbolic pol}
          \beta \wedge \alpha_1 \wedge \ldots\wedge \alpha_m \wedge \theta^{n-m} \geq 0,
    \end{equation}
    for all $(\theta,m+1)$-positive $(1,1)$-forms $\alpha_1,\ldots,\alpha_m$.
\end{lemma}
\begin{proof}
 We apply Gårding's theory of hyperbolic polynomials, with a detailed discussion available in \cite{Gar59}. Due to the form $\theta$ not being strictly positive, an approximation argument is necessary. Given that the inequality is pointwise, we can assume all $(1,1)$-forms involved have constant coefficients. These are identifiable with Hermitian matrices in $\mathbb{C}^{n+1}$, which correspond to $\mathbb{R}^{(n+1)^2}$.

 A homogeneous polynomial $P$ of degree $m$ is hyperbolic with respect to a $(1,1)$-form $\beta$ if for any $(1,1)$-form $\alpha$, the equation $P(\alpha+t\beta)=0$ has $m$ real solutions. If $dV$ denotes the standard Euclidean volume forme, then the polynomial $Q(\alpha)= \alpha^{n+1}/dV$ is hyperbolic with respect to any strictly positive $(1,1)$-form, and the cone $C(Q)$ consists of those forms (see Example 4 and page 960 in \cite{Gar59}). Let $M$ be the completely polarized form of $Q$, defined as
 \[
 M(\alpha_1,\ldots,\alpha_{n+1}) = \frac{\alpha_1\wedge\ldots \wedge \alpha_{n+1}}{dV}.
 \]
 For each $\varepsilon>0$, by \cite[Theorem 4]{Gar59}, the polynomial $P_{\varepsilon}(\alpha):= M(\alpha,\ldots,\alpha, \theta+\varepsilon \omega,\ldots,\theta+\varepsilon\omega)$ (with $\alpha$ repeated $m+1$ times), is hyperbolic with respect to any strictly positive $(1,1)$-form, particularly $\omega$. Consequently, the equation $P_{\varepsilon}(\alpha + t \omega) = 0$ has $m+1$ real solutions. As $\varepsilon \to 0^+$, it follows that $P(\alpha + t \omega) = 0$ also has $m+1$ real solutions, establishing that $P(\alpha)=\alpha^{m+1}\wedge \theta^{n-m}/dV$ is hyperbolic with respect to $\omega$.

Recall from \cite{Gar59} that the cone $C(P, \omega)$ comprises all $\beta$ such that $P(\beta + t \omega) > 0$ for all $t \geq 0$ (since $P(\omega)>0$). Consequently, any $(\theta, m+1)$-positive form resides within the closure of $C(P, \omega)$. Hence, by \cite[Theorem 5]{Gar59}, $M(\beta, \alpha_1,\ldots, \alpha_m) \geq 0$ if $\beta, \alpha_1,\ldots, \alpha_m$ are $(\theta, m+1)$-positive, where $M$ is the completely polarized form of $P$, expressed by the wedge product:
 \begin{equation*}
M(\beta, \alpha_1, \ldots, \alpha_m) = \frac{\beta \wedge \alpha_1 \wedge \cdots \wedge \alpha_m \wedge \theta^{n-m}}{dV}.
\end{equation*}

Suppose $\beta$ satisfies \eqref{eq: hyperbolic pol} for all $(\theta, m+1)$-positive forms $\alpha_1,\ldots, \alpha_m$, yet by contradiction, $\beta$ is not $(\theta, m+1)$-positive. Let $t$ be the infimum of $s>0$ such that $\beta + s\omega$ is $(\theta, m+1)$-positive. Then $t>0$ 
and there exists $1\leq k\leq m+1$ such that
\[
(\beta + t\omega)^k \wedge \omega^{m+1-k}\wedge \theta^{n-m}=0.
\]
For, if they are all positive, then for some $t_1>0$ slightly smaller than $t$, the form $\beta+t_1\omega$ is still $(\theta,m+1)$-positive, contradicting the definition of $t$. Expanding this equality as
\[
\beta \wedge (\beta+t\omega)^{k-1}\wedge \omega^{m+1-k} \wedge \theta^{n-m} + t\omega \wedge (\beta+t\omega)^{k-1}\wedge \omega^{m+1-k} \wedge \theta^{n-m} =0,
\]
and using \eqref{eq: hyperbolic pol} we infer that the two terms above, which are non-negative, must be zero. In particular, looking at the second term, and using $t>0$, we arrive at $(\beta+t\omega)^{k-1}\wedge \omega^{m+2-k}\wedge \theta^{n-m}=0$. Repeating this argument, we arrive at
$P(\omega) = 0$, which is a contradiction.
\end{proof}

\begin{corollary}
     Let $\alpha$ be a $(1,1)$-form in $\Omega$. Then the $(1,1)$-form $\pi^*\alpha$ is $(\theta, m+1)$-positive in $D$ if and only if $\alpha$ is $m$-positive in $\Omega$.
\end{corollary}
\begin{proof}
   Let $\alpha= \sum_{j,k=1}^n \alpha_{j,k} dz_j \wedge d\bar{z}_k$ be a $(1,1)$-from in $\Omega$. Observe that
   \begin{flalign*}
        (\pi^*\alpha)^k &\wedge \omega^{m+1-k}\wedge \theta^{n-m}= (\pi^*\alpha)^k \wedge (\theta+dd^c |w|^2)^{m+1-k}\wedge \theta^{n-m}\\
        &= \sum_{l=0}^{m+1-k} \binom{m+1-k}{l}(\pi^* \alpha)^k \wedge \theta^{n+l-m}  \wedge (dd^c |w|^2)^{m+1-k-l}\\
        &= \binom{m+1-k}{m-k}(\pi^* \alpha)^k \wedge \theta^{n-k}  \wedge (dd^c |w|^2).
   \end{flalign*}
   Here we use that the term corresponding to $m+1-k-l =0$ is zero because an $(n+1,n+1)$-form involving only $dz_j$ terms is zero, while  the term corresponding to $m+1-k-l \geq 2$ is zero because $(dd^c |w|^2)^2=0$. Thus, if $\alpha$ is $m$-positive in $\Omega$, then from the expansion above, we see that it is non-negative for $k\leq m$, while $(\pi^*\alpha)^{m+1}\wedge \theta^{n-m}$ vanishes.

   If $\pi^*\alpha$ is $(\theta,m+1)$-positive then, in Lemma \ref{GC: Lem}, taking $\alpha_1=dd^c |w|^2$ and $\alpha_2,\ldots,\alpha_m$ $m$-positive forms in $\Omega$, we see that $\alpha$ is $m$-positive.
\end{proof}

\begin{definition}
   A function $u$ is $(\theta, m+1)$-subharmonic in $D$ if it is subharmonic in $D$ and
\[
dd^c u \wedge \alpha_1 \wedge \ldots \wedge \alpha_m \wedge \theta^{n-m} \geq 0
\]
in the weak sense of currents, for all $(\theta, m+1)$-positive $(1,1)$-forms $\alpha_1, \ldots, \alpha_m$.
\end{definition}

We impose subharmonicity here to ensure that two $(\theta,m+1)$-subharmonic functions coinciding almost everywhere are equal everywhere.

\begin{remark}
Utilizing Lemma~\ref{GC: Lem}, we see that a $C^2$-function $u$ is $(\theta, m+1)$-subharmonic if and only if the $(1,1)$-form $dd^c u$ is $(\theta, m+1)$-positive.
\end{remark}

\begin{remark}
Taking $\alpha_1 = dd^c |w|^2$ and $\alpha_2, \ldots, \alpha_m$ as $m$-positive $(1,1)$-forms in $\Omega$ in the above definition, we observe that the slice $z \mapsto u(z, w)$ is $m$-subharmonic in $\Omega$ for each fixed $w$. Similarly, by taking $\alpha_j = \theta$ for all $j$, we deduce that $w \mapsto u(z, w)$ is subharmonic in $A$ for any fixed $z$. In particular, if $u$ is $S^1$-invariant in $w$, then the function $t\mapsto u(z,e^t)$ is convex in $[0,1]$.
\end{remark}

Below, we summarize the basic properties of $(\theta, m+1)$-subharmonic functions, the proofs of which are straightforward adaptations from the case when $m = n$.

\begin{proposition}{\; }
    \begin{enumerate}
        \item If $u,v$ are $(\theta,m+1)$-subharmonic in $\Omega$ then $\max(u,v)$, $au+bv$ are $(\theta,m+1)$-subharmonic in $\Omega$ for any $a,b>0$.
        \item  If $u$ is $(\theta,m+1)$-subharmonic then the standard regularization $u\star \chi_{\varepsilon}$ are also $(\theta,m+1)$-subharmonic and $u\star \chi_{\varepsilon}\searrow u$.
        \item If $(u_j)_{j\in J}$ is a family of $(\theta,m+1)$-subharmonic functions which are locally uniformly bounded from above then $(\sup_{j\in J}u_j)^*$ is also $(\theta,m+1)$-subharmonic. Here $*$ stands for the upper semicontinuous regularization.
        \item If $m=n$ then $(\theta,n+1)$-subharmonic means plurisubharmonic.
    \end{enumerate}
\end{proposition}

\subsection{$m$-subharmonic geodesics}
Since the influential publication of Mabuchi's seminal work on constant scalar curvature Kähler metrics \cite{Mab87}, which introduced the concept of plurisubharmonic geodesics, there has been vigorous activity in the mathematical community. From our viewpoint, it is critical to represent geodesics as the upper envelopes of sub-geodesics. For $n=m$, this approach has been adopted in a local context by Berman-Berndtsson \cite{BB22}, Abja \cite{Abj19}, Abja-Dinew \cite{AD21,AD24}, and Rashkovskii \cite{Ras17}. In the subsequent discussion, we introduce an analogous concept of sub-geodesics for $m$-subharmonic functions.

\begin{definition}
  A curve $u_t \in \MSH{\Omega}$, for $t \in (0,1)$, is  an \emph{$m$-subharmonic sub-geodesic} if the function $U(z,w) = u_{\log |w|}(z)$ is $(\theta, m+1)$-subharmonic in $D$.
\end{definition}
 Any function $u \in \MSH{\Omega}$, when viewed as a function in $D$, is $(\theta, m+1)$-subharmonic. The mappings $(z,t) \mapsto \pm t$ are $(\theta, m+1)$-subharmonic because their associated functions $U(z,w) = \pm \log |w|$ are plurisubharmonic in $D$, and thus also $(\theta, m+1)$-subharmonic. In particular, if $u_0, u_1 \in \MSH{\Omega}$ and $C$ is a constant, then the function
\[
v_t = \max(u_0 - Ct, u_1 + Ct - C)
\]
is an $m$-subharmonic sub-geodesic. If $C \geq \sup_{\Omega} |u_0 - u_1|$, then $v_t$  connects $u_0$ to $u_1$, in the sense that
\[
\lim_{t \to 0} v_t = u_0 \; \text{and} \; \lim_{t \to 1} v_t = u_1,
\]
pointwise in $\Omega$. We define $\mathcal{S}(u_0, u_1)$ as the set of all $m$-subharmonic sub-geodesics $v_t$ that lie below $u_0$ and $u_1$ in the sense that
\[
\limsup_{t \to 0} u_t(z) \leq u_0(z), \quad \limsup_{t \to 1} u_t(z) \leq u_1(z),
\]
for almost every $z \in \Omega$. The non-emptiness of the set $\mathcal{S}(u_0, u_1)$ is ensured, as $u_0 + u_1$ itself is a member. Each $m$-subharmonic sub-geodesic $v_t$ is subharmonic and $S^1$-invariant in $w = e^t$, thus the map $t \mapsto v_t$ is convex. If $v_t \in \mathcal{S}(u_0, u_1)$, then it satisfies $v_t \leq (1-t)u_0 + t u_1$.
\begin{definition}
Given $u_0, u_1 \in \MSH{\Omega}$, we define the \emph{$m$-subharmonic geodesic between $u_0$ and $u_1$} as
\[
u_t(z) = \sup \{v_t(z) : v \in \mathcal{S}(u_0, u_1)\}.
\]
\end{definition}
Since $u_t \leq (1-t)u_0 + t u_1$, the upper semicontinuous regularization of $(z,t) \mapsto u_t(z)$ also belongs to $\mathcal{S}(u_0, u_1)$, confirming that $u_t$ is indeed an element of $\mathcal{S}(u_0, u_1)$. The convexity of $u_t$ further implies that for bounded functions $u_0, u_1 \in \MSH{\Omega}$, the $m$-subharmonic geodesic $u_t$ satisfies
\[
\lim_{t \to 0} u_t(z) = u_0(z), \quad \text{and} \quad \lim_{t \to 1} u_t(z) = u_1(z),
\]
for all $z \in \Omega$. For unbounded functions $u_0, u_1$, it is preferable to permit convergence almost everywhere. We say that \emph{$u_0$ can be connected to $u_1$ by an $m$-subharmonic geodesic} if these limits hold almost everywhere in $\Omega$.

\begin{remark}\label{rem:approximation}
Given negative functions $u_0, u_1 \in \MSH{\Omega}$, we consider their approximants $u_0^j, u_1^j \in \MSH{\Omega} \cap L^{\infty}$, which decrease to $u_0$ and $u_1$ respectively. Let $u_t^j$ denote the $m$-subharmonic geodesics connecting $u_0^j$ to $u_1^j$. As $j$ increases, $u_t^j$ decreases and converges to $u_t = \lim_{j\to+\infty} u_t^j$. This limit $u_t$ represents the largest $m$-subharmonic sub-geodesic lying below $u_0$ and $u_1$; notably, it is independent of the specific approximants $u_0^j, u_1^j$.  Indeed, if $v_t$ denotes the $m$-subharmonic geodesic between $u_0$ and $u_1$, then $v_t \leq u_t^j$ for all $j$, and consequently, $v_t \leq u_t$. Conversely, since $u_t^j \leq (1-t)u_0^j + t u_1^j$ for all $j$ due to convexity, taking the limit as $j \to \infty$ implies that $u_t$ belongs to $\mathcal{S}(u_0, u_1)$ and thus satisfies $u_t \leq v_t$.
\end{remark}

\subsection{The Kiselman minimum principle}{\; }

\bigskip

\noindent {\bf Theorem~D.} \emph{Let $(u_t)_{t\in (0,1)}$ be an $m$-subharmonic sub-geodesic in $\Omega$. Then the function $\inf_{t\in (0,1)} u_t$ is $m$-subharmonic in $\Omega$.}

\bigskip

Just as in the scenario where $m = n$, the infimum of a family of $m$-subharmonic functions is often not $m$-subharmonic. The importance of the sub-geodesic assumption, particularly the $S^1$-invariant property, will become evident in the proof presented below. Our argument is inspired by a direct computation in \cite[page 63]{GZbook}. It is different from the original proof of Kiselman which uses a crucial property of plurisubharmonic functions (the restriction on each complex line is subharmonic) that does not hold for $m$-subharmonic functions when $m<n$.

\begin{proof}

We begin by assuming that $(t,z) \mapsto u_t(z)$ is smooth and define the functions
\[
v(z,t) = u_{t}(z) - \varepsilon \zeta(t), \; V(z,w)=v(z,\log |w|),
\]
where $\varepsilon > 0$ is fixed, and $\zeta(t)=\log(t(1-t))$ is a  concave function.
Since $w\mapsto \log |w|$ is harmonic in the annulus $A=\{1<|w|<e\}$, it follows that $w\mapsto -\zeta(\log |w|)$ is a strictly subharmonic function in $A$, as the following computation shows
$$dd^c \zeta (\log |w|) = \zeta''(\log |w|)  d \log |w| \wedge d^c \log |w|<0.$$
Note also that $V(z,w) \to +\infty$ as $w \to \partial A$.  Due to the strict convexity of $t \mapsto v(z,t)$, we infer that for each $z \in \Omega$, there exists a unique $t(z) \in (0,1)$ such that $v(z,t(z)) = \inf_{s \in (0,1)} v_s(z)$. Furthermore, we have
\begin{equation}
    \label{eq: Kis}
    \frac{\partial v}{\partial t}(z,t(z)) = 0, \quad \text{ for all } z\in \Omega.
\end{equation}
Since, $\frac{\partial^2 v}{\partial t^2}(z,t(z)) > 0$, the implicit function theorem implies that $z \mapsto t(z)$ is smooth. The Hessian of $z \mapsto h(z) = v(z,t(z))$ is given by
 \[
 \frac{\partial h(z)}{\partial \bar z_k} = \frac{\partial v}{\partial \bar z_k} (z, t(z)) + \frac{\partial v}{\partial t}(z,t(z)) \frac{\partial t}{\partial \bar z_k}(z),
 \]
 and
\begin{flalign}
\frac{\partial^2 h(z)}{\partial z_j \partial \bar{z}_k}  &= \frac{\partial^2 v}{\partial z_j \partial \bar{z}_k} (z,t(z)) + \frac{\partial^2 v}{\partial t^2}(z,t(z)) \frac{\partial t}{\partial \bar z_k}(z) \frac{\partial t}{\partial z_j}(z) \nonumber\\
&+ \frac{\partial^2 v}{\partial t \partial \bar z_k}(z,t(z)) \frac{\partial t(z)}{\partial z_j} + \frac{\partial^2 v}{\partial t \partial z_j}(z,t(z)) \frac{\partial t(z)}{\partial \bar z_k}. \label{eq: Hes of h}
\end{flalign}
Here, we have used \eqref{eq: Kis} to obtain $\frac{\partial v}{\partial t}(z,t(z)) \frac{\partial^2 t}{\partial z_j \partial \bar z_k} =0$.
For $\xi = (\xi_1,\ldots,\xi_{n})\in \mathbb{C}^{n}$, we consider the positive $(1,1)$-form
\[
\alpha :=  i \gamma \wedge \bar \gamma, \; \text{where}\; \gamma =  \sum_{j=1}^n \xi_j dz_j - dw.
\]
In $\Omega \times A$, we compute
\[
dd^c V \wedge \alpha =  \beta_1 + \beta_2 \wedge i dw \wedge d\bar w,
\]
where $\beta_1$ is a $(2,2)$-form which does not contain the term $dw\wedge d\bar w$, and
\[
\beta_2 =  i\sum_{j,k=1}^n \left (\frac{\partial^2 V}{\partial z_j \partial \bar{z}_k}  + \frac{\partial^2 V}{\partial z_j \partial \bar w} \bar{\xi}_k+ \frac{\partial^2 V}{\partial w \partial \bar{z}_k} \xi_j+ \xi_j \bar{\xi}_k \frac{\partial^2 V}{\partial w \partial \bar w} \right) dz_j \wedge d\bar{z}_k.
\]
Let $\alpha_1,\ldots,\alpha_{m-1}$ be $m$-positive $(1,1)$-forms with constant coefficients in $\Omega$. Since the form $\alpha_1 \wedge \ldots \wedge \alpha_{m-1} \wedge \theta^{n-m}$ does not contain $dw$ nor $d\bar w$, by the above computation, we get
\[
dd^c V \wedge \alpha \wedge \pi^*\alpha_1 \wedge \ldots\wedge \pi^*\alpha_{m-1} \wedge \theta^{n-m} = \beta_2 \wedge \pi^*\alpha_1 \wedge \ldots \wedge \pi^*\alpha_{m-1} \wedge \theta^{n-m} \wedge (idw \wedge d\bar w).
\]
Since $V$ is $(m+1,\theta)$-subharmonic, we infer that the left-hand side is positive; hence so is the right-hand side. This yields $\beta_2 \wedge \alpha_1 \wedge \ldots \wedge \alpha_{m-1} \wedge \theta^{n-m} \geq 0$, and since the $m$-positive $(1,1)$-forms $\alpha_1,\ldots,\alpha_{m-1}$ was chosen arbitrarily, we infer that $\beta_2$ is $m$-positive. Recall that
\[
\frac{\partial V}{\partial w} = \frac{\partial v}{\partial t} \frac{1}{2w}, \; \frac{\partial^2 V}{\partial w \partial \bar w} = \frac{\partial^2 v}{\partial t^2} \frac{1}{4|w|^2}.
\]
Choosing $\xi_j = 2w \frac{\partial t}{\partial z_j}(z)$, for $j=1,\ldots,n$, we infer from the $m$-positivity of $\beta_2$ and \eqref{eq: Hes of h} that $dd^c h$ is $m$-positive, hence $h$ is $m$-subharmonic in $\Omega$. As $\varepsilon \to 0^+$, we conclude that $\inf_{t \in (0,1)} u_t(z)$ is $m$-subharmonic, being the decreasing limit of such functions.

To address the general case, we approximate $(z,w) \mapsto u_{\log |w|}(z)$ by convolution in $\Omega \times A$, yielding a decreasing sequence $u_j$ of smooth $(\theta, m+1)$-subharmonic functions defined in a slightly smaller domain $\Omega' \times A'$, which remain $S^1$-invariant in $w = e^t$. The function $\varphi_{j,\varepsilon} = \inf_{t \in (\varepsilon, 1-\varepsilon)} u_j(z,e^t)$ is thus $m$-subharmonic in $\Omega'$. As $j \to +\infty$, $\varphi_{j,\varepsilon} \searrow \varphi_{\varepsilon} = \inf_{t \in (\varepsilon, 1-\varepsilon)} u(z,e^t)$, confirming $\varphi_{\varepsilon}$ as $m$-subharmonic in $\Omega'$. Finally, as $\varepsilon \to 0^+$ and $\Omega' \to \Omega$, we confirm the result since $\varphi_{\varepsilon} \searrow \inf_{t \in (0,1)} u_t$.
\end{proof}

\begin{corollary}\label{cor: inf of geo}
    If $u_t$ is an $m$-subharmonic geodesic between $u_0$ to $u_1$, then
\[
\inf_{t\in (0,1)} u_t = \PP(u_0,u_1).
\]
\end{corollary}
\begin{proof}
Applying the Kiselman minimum principle, we establish that $v = \inf_{t \in (0,1)} u_t$ is an $m$-subharmonic function that lies below $u_t$ for any $t \in (0,1)$. Given the assumption $\limsup_{t \to a} u_t \leq u_a$ for $a = 0, 1$, almost everywhere in $\Omega$, it naturally follows that $v \leq \PP(u_0, u_1)$ (a.e., and hence everywhere). Moreover, since the curve $w_t = P(u_0, u_1)$, $t\in (0,1)$ is a candidate defining $u_t$, we infer that $P(u_0, u_1) \leq u_t$.
\end{proof}

\subsection{Geodesic connectivity}\label{sec:Geodesic}
 In our previous work \cite{ACLR24}, we prove that given $H_0, H_1 \in \mathcal{E}_m$ that are connectable by a plurisubharmonic geodesic, and $u_0 \in \Ker(H_0)$, $u_1 \in \Ker(H_1)$ (with $m=n$), then $u_0$ and $u_1$ can be connected by a plurisubharmonic geodesic segment if and only if
\begin{equation*}
    u_0 \leq g_{u_1} \quad \text{and} \quad u_1 \leq g_{u_0},
\end{equation*}
where $g_{u_i}$, $i=0,1$, is the residual function of $u_i$ as defined in Definition~\ref{def: residual function}. The goal of this section is to extend this result to the case where $m<n$.

\begin{theorem}\label{thm: geo connect Darvas}
    Assume $u_0,u_1$ are $m$-subharmonic  functions in $\Omega$. Then $u_0$ can be connected to $u_1$ by an $m$-subharmonic  geodesic if and only if
    \begin{equation}
        \label{eq: geo connectivity}
        \PP[u_0](u_1)=u_1 \quad \text{and}\quad \PP[u_1](u_0)=u_0.
    \end{equation}
\end{theorem}
\begin{proof}
The proof adapts the argument presented by Darvas for the case $m = n$, made possible by the successful application of the Kiselman minimum principle (Theorem~D). For clarity, we briefly recall Darvas' approach from \cite{Dar17AJM}.

Assume that \eqref{eq: geo connectivity} is satisfied. For each $C > 0$, consider the curve $v_t = \PP(u_0, u_1 + C) - Ct$, which serves as an $m$-subharmonic sub-geodesic connecting $v_0 \leq u_0$ to $v_1 \leq u_1$. Consequently, we have $v_t \leq u_t$, and specifically for $t = 1/C^2$, it follows that
\[
\PP(u_0, u_1 + C) - 1/C \leq u_{1/C^2}.
\]
As $C \to +\infty$, this results in $\PP[u_1](u_0) \leq \liminf_{t \to 0} u_t$ almost everywhere. The convexity of $u_t$ gives $u_t \leq (1-t)u_0 + tu_1$, leading to $\limsup_{t \to 0} u_t \leq u_0$ and hence $\lim_{t \to 0} u_t = u_0$, almost everywhere. A similar argument applies for the limit at $t = 1$.

Conversely, assume $u_t$ connects $u_0$ to $u_1$. Fix a constant $C > 0$, and consider the curve $v_t = u_t + tC$, which is an $m$-subharmonic geodesic connecting $u_0$ to $u_1 + C$. By Corollary \ref{cor: inf of geo}, $\PP(u_0, u_1 + C) = \inf_{t \in (0,1)} (u_t + Ct)$. Due to the convexity of $t \mapsto u_t$ (referenced in \cite[Lemma 5.1]{Dar17AJM}), we find that
\[
\lim_{C \to +\infty} \inf_{t \in (0,1)} (u_t + Ct) = \liminf_{t \to 0} u_t,
\]
outside the $m$-polar set $\{\min(u_0, u_1) = -\infty\}$. Since $\liminf_{t \to 0} u_t = u_0$ almost everywhere, allowing $C \to +\infty$ confirms that $\PP[u_1](u_0) = u_0$ almost everywhere, hence everywhere.
The equality $\PP[u_0](u_1)=u_1$ follows by reversing the role of $u_0$ and $u_1$. This concludes the proof.
\end{proof}

Building on Theorem~\ref{thm: geo connect Darvas} and the rooftop equality (Theorem~B), we can adapt the proof of \cite{ACLR24} to the $m$-subharmonic setting.

\bigskip

\noindent {\bf Theorem~C.}\emph{
    Assume $H_0,H_1 \in \mathcal{E}_m$ are connectable by an $m$-subharmonic geodesic segment. Let $u_0\in \Ker(H_0)$, $u_1\in \Ker(H_1)$. Then $u_0$ and $u_1$ are connectable by an $m$-subharmonic geodesic if and only if
    \begin{equation}
        \label{eq: geo connect}
        g_{u_0}\geq u_1 \; \text{and} \; g_{u_1}\geq u_0.
    \end{equation}
    In particular, if $g_{H_0}=g_{H_1}$ then \eqref{eq: geo connect} is equivalent to $g_{u_0}=g_{u_1}$.
    }
\begin{proof}
From Theorem~B and Theorem~\ref{thm: geo connect Darvas}, it is established that the rooftop equality holds for $u_0$ and $u_1$. Consequently, \eqref{eq: geo connectivity} is equivalent to both $P(g_{u_0}, u_1) = u_1$ and $P(g_{u_1}, u_0) = u_0$, which in turn is equivalent to \eqref{eq: geo connect}. We confirm the first statement by applying Theorem~\ref{thm: geo connect Darvas} again.

Assuming that \eqref{eq: geo connect} holds, let us consider $v = \PP(g_{u_0}, g_{u_1})$. Theorems~\ref{thm: MA measure of asymptotic envelope} and \ref{cor: NP of rooftop} ensure that $\mu_r(v) = 0$. Utilizing Lemma~\ref{lem: ACCP09 lemma 4.1} and again Theorem~\ref{thm: MA measure of asymptotic envelope}, we deduce that $\mu_s(v) \geq \mu_s(g_{u_0})$ and $\mu_s(v) \geq \mu_s(g_{u_1})$. The application of the general comparison principle (Theorem~\ref{thm: uniqueness in N}) then establishes that $v = g_{u_0} = g_{u_1}$, thereby completing the proof.
\end{proof}

\end{document}